\documentclass[preprint]{imsart}

\pdfoutput=1
\RequirePackage[OT1]{fontenc}
\usepackage{amsthm, amsmath, natbib, amssymb, bm, enumitem}
\RequirePackage[colorlinks,citecolor=blue,urlcolor=blue]{hyperref}

\usepackage{mathtools} 
\mathtoolsset{showonlyrefs}
\usepackage{algorithm,algorithmic}

\startlocaldefs
\numberwithin{equation}{section}
\theoremstyle{plain}
\newtheorem{theorem}{Theorem}[section]
\newtheorem{lemma}{Lemma}[section]

\theoremstyle{remark}

\newcommand{\rd}{\mathrm{d}}

\endlocaldefs

\begin{document}
\begin{frontmatter}
\title{Minimaxity under the half-Cauchy prior}
\runtitle{Minimaxity under the half-Cauchy prior}

\begin{aug}
\author{\fnms{Yuzo} \snm{Maruyama}\thanksref{addr1,t1}\ead[label=e1]{maruyama@port.kobe-u.ac.jp}}
\and
\author{\fnms{Takeru} \snm{Matsuda}\thanksref{addr2,t2}\ead[label=e2]{matsuda@mist.i.u-tokyo.ac.jp}}

\runauthor{Y.~Maruyama and T.~Matsuda}

\address[addr1]{Kobe University \printead{e1} 
}

\address[addr2]{The University of Tokyo \& RIKEN Center for Brain Science \printead{e2}
}

\thankstext{t1}{supported by JSPS KAKENHI Grant Numbers 19K11852 and 22K11933}
\thankstext{t2}{supported by JSPS KAKENHI Grant Numbers 21H05205, 22K17865 and JST Moonshot Grant Number JPMJMS2024}
\end{aug}

\begin{abstract}
This is a follow-up paper of Polson and Scott (2012, Bayesian Analysis),
which claimed that the half-Cauchy prior is a sensible default prior for a scale parameter in hierarchical models. 
For estimation of a $p$-variate normal mean under the quadratic loss, they demonstrated that the Bayes estimator with respect to the half-Cauchy prior seems to be minimax through numerical experiments.
In this paper, we theoretically establish the minimaxity of the corresponding Bayes estimator using the interval arithmetric.
\end{abstract}



\end{frontmatter}

\section{Introduction}
\label{sec:intro}
Consider a normal hierarchical model 
\begin{equation}\label{hierarchical.1}
y\mid\beta\sim \mathcal{N}_p(\beta,I_p), \quad  \\
\beta\mid\kappa \sim \mathcal{N}_p\bigl(0,\frac{1-\kappa}{\kappa}I_p\bigr), \quad
\kappa \sim \pi (\kappa), 
\end{equation}
where the hyperparameter $\kappa \in (0,1)$ specifies the shrinkage coefficient of 
posterior mean of $\beta$:
\begin{equation}\label{posterior.mean}
	\hat{\beta}(y) = {\rm E} [\beta \mid y] = (1-{\rm E}[\kappa \mid y]) y.
\end{equation}
\cite{Polson-Scott-2012} claimed that the hyperprior
\begin{equation}\label{prior.half}
 \pi(\kappa) \propto \kappa^{-1/2}(1-\kappa)^{-1/2},
\end{equation}
is a sensible default choice.
Since it has a U-shape with
\begin{equation}\label{cont.spike.slab}
 \lim_{\kappa\to 0}\pi(\kappa)=\lim_{\kappa\to 1}\pi(\kappa)=\infty,
\end{equation}
it may be regarded as a continuous spike and slab prior.
See \cite{carvalho2010horseshoe} for a related discussion in the context of horseshoe priors.
For the parameterization
\begin{equation}
 \lambda=\sqrt{\frac{1-\kappa}{\kappa}} \in (0,\infty),
\end{equation}
the prior \eqref{prior.half} is expressed as
\begin{equation}
 \pi(\lambda) \propto \frac{1}{1+\lambda^2}I_{(0,\infty)}(\lambda),
\end{equation}
which is the reason why the prior \eqref{prior.half} is called the half-Cauchy prior.

For estimation of a $p$-variate normal mean $\beta$ under the quadratic loss $\|\hat{\beta}-\beta\|^2$, 
the posterior mean \eqref{posterior.mean} is the minimizer of the corresponding Bayes risk and given by
\begin{equation}\label{hat.beta.0}
 \hat{\beta}
=\left(1-\frac {\int_0^1\kappa^{p/2+1/2}(1-\kappa)^{-1/2}\exp(-\kappa\|y\|^2/2)\rd\kappa}
{\int_0^1\kappa^{p/2-1/2}(1-\kappa)^{-1/2}\exp(-\kappa\|y\|^2/2)\rd\kappa}
\right)y.
\end{equation}
\cite{Polson-Scott-2012} derived 
expressions for the risk of the Bayes estimator \eqref{hat.beta.0}.
Recall that the usual estimator $\hat{\beta}=y$ is inadmissible for $p\geq 3$ although it is minimax for any $p$ \citep{Stein-1956, DSW-2018}.
Through numerical experiments, \cite{Polson-Scott-2012} discussed the minimaxity of 
the Bayes estimator \eqref{hat.beta.0} and compared it with the \cite{James-Stein-1961} estimator.

In this paper, we theoretically establish the minimaxity 
of the Bayes estimator \eqref{hat.beta.0} for $p\geq 7$, as follows.
\begin{theorem}\label{thm:main}
The Bayes estimator \eqref{hat.beta.0} under the half-Cauchy prior \eqref{prior.half} 
is minimax for $p\geq 7$.
\end{theorem}
In the proof of Theorem \ref{thm:main}, we employ the interval arithmetic \citep{Moore-2009}, which has been used in the field of verified numerical computation.
See Appendix~\ref{sec:interval} for a brief review of the interval arithmetic.
To our knowledge, the interval arithmetic has not been common in the community of statisticians.

The organization of the paper is as follows.
We start with a more general prior 
\begin{equation}\label{our.prior}
 \pi(\kappa)= \kappa^{a-1}(1-\kappa)^{b-1}\text{ with }0<b<1.
\end{equation}
Note, for $b\geq 1$, 
the minimaxity of the corresponding (generalized) Bayes estimators has been well investigated
by \cite{Berger-1976} and \cite{Faith-1978}.
\cite{Maruyama-1998} treated the case $0<b<1$ and showed that
the corresponding Bayes estimators with $(p+2a+2)/(3p/2+a)\leq b<1$ are minimax.
However, these results do not cover the case $a=b=1/2$ (half-Cauchy prior) for any $p\geq 3$.
In Section \ref{sec:Bayes_SURE}, we give the Bayes estimators under \eqref{our.prior}
and their Stein's unbiased risk estimates.
In Section \ref{sec:minimax.general}, we provide a sufficient condition for the Bayes estimators under \eqref{our.prior}
to be minimax, which will be stated as Theorem \ref{thm:main_general}.
Then, in Section \ref{subsec:proof.2}, we focus on the half-Cauchy prior, which is a special case of \eqref{our.prior} with $a=b=1/2$, and prove Theorem \ref{thm:main}. 
Whereas the minimaxity of the half-Cauchy prior for $p\geq 11$ directly follows from Theorem \ref{thm:main_general}, the case of $7\leq p\leq 10$ requires additional investigation using the interval arithmetic.
Many of the proofs of technical lemmas are given in Appendix.
The python code for the interval arithmetic proof is available at \url{https://github.com/takeru-matsuda/half_cauchy}.

\section{Bayes estimators and Stein's unbiased risk estimate}
\label{sec:Bayes_SURE}
By the identity
\begin{equation}
 \|y-\beta\|^2+\frac{\kappa}{1-\kappa}\|\beta\|^2
=\frac{1}{1-\kappa}\left\|\beta-(1-\kappa)y\right\|^2
+\kappa\|y\|^2,
\end{equation}
the marginal density under the model \eqref{hierarchical.1} is
\begin{align}
 m(y)&=\iint \frac{1}{(2\pi)^{p/2}}\exp\left(-\frac{\|y-\beta\|^2}{2}\right)\pi(\beta\mid\kappa)\pi(\kappa)\rd \beta \rd\kappa\\
&=\iint \frac{1}{(2\pi)^{p/2}}\exp\left(-\frac{\|y-\beta\|^2}{2}\right)
\frac{1}{(2\pi)^{p/2}}\left(\frac{\kappa}{1-\kappa}\right)^{p/2}\notag\\
&\quad\times\exp\left(-\frac{\kappa}{1-\kappa}\frac{\|\beta\|^2}{2}\right)
\kappa^{a-1}(1-\kappa)^{b-1}
\rd \beta\rd\kappa \notag\\
&=\frac{1}{(2\pi)^{p/2}}\int \kappa^{p/2+a-1}(1-\kappa)^{b-1}\exp\left(-\kappa\frac{\|y\|^2}{2}\right)
\rd\kappa. \label{my.1}
\end{align}
Note $\exp(-\kappa\|y\|^2/2)\leq 1$. Then, for 
\begin{equation}\label{a-b-hanni}
 p/2+a>0\text{ and }b>0,
\end{equation}
we have
\begin{equation}
 m(y)\leq \frac{1}{(2\pi)^{p/2}}
\int_0^1\kappa^{p/2+a-1}(1-\kappa)^{b-1}\rd\kappa=
\frac{\mathrm{B}(p/2+a,b)}{(2\pi)^{p/2}}<\infty.\label{marginal.0}
\end{equation}
By Tweedie's formula \citep{Efron-2011, Efron-2023-jjsd}, 
the Bayes estimator is
\begin{equation}\label{hat.beta}
 \hat{\beta}
=y+\nabla \log m(y) 
=\left(1-\frac {\int_0^1\kappa^{p/2+a}(1-\kappa)^{b-1}\exp(-\kappa\|y\|^2/2)\rd\kappa}
{\int_0^1\kappa^{p/2+a-1}(1-\kappa)^{b-1}\exp(-\kappa\|y\|^2/2)\rd\kappa}
\right)y.
\end{equation}
Since the marginal density $m(y)$ given by \eqref{my.1} is spherically symmetric, 
let $m_*(\|y\|^2) \coloneqq m(y)$.
Then, the quadratic risk of the Bayes estimator \eqref{hat.beta} is
\begin{equation}
 \begin{split}
&  E\bigl[\|\hat{\beta}-\beta\|^2\bigr] =E\biggl[\Bigl\|Y-2\frac{m'_*(\|Y\|^2)}{m_*(\|Y\|^2)}Y-\beta
\Bigr\|^2\biggr]\\
&=E\biggl[\|Y-\beta\|^2+4\left(\frac{m'_*(\|Y\|^2)}{m_*(\|Y\|^2)}\right)^2\|Y\|^2
-4\sum_{i=1}^p(Y_i-\beta_i)Y_i\frac{m'_*(\|Y\|^2)}{m_*(\|Y\|^2)}\biggr]\\
&=E\left[p+4\frac{m'_* (\|Y\|^2)}{m_* (\|Y\|^2)}
\Bigl(p-2\|Y\|^2\frac{m''_*(\|Y\|^2)}{-m'_*(\|Y\|^2)}
+\|Y\|^2\frac{-m'_*(\|Y\|^2)}{m_*(\|Y\|^2)}\Bigr)\right],
 \end{split}
\end{equation}
where the third equality follows from \cite{Stein-1974} identity.
Let
\begin{equation}\label{hat.R.0}
 \hat{R}(\|y\|^2)=p+4\frac{m'_* (\|y\|^2)}{m_* (\|y\|^2)}
\Bigl(p-2\|y\|^2\frac{m''_*(\|y\|^2)}{-m'_*(\|y\|^2)}
+\|y\|^2\frac{-m'_*(\|y\|^2)}{m_*(\|y\|^2)}\Bigr).
\end{equation}
Then we have $ E[\|\hat{\beta}-\beta\|^2] =E[\hat{R}(\|Y\|^2)]$ and
$ \hat{R}(\|y\|^2)$ is called the Stein's unbiased risk estimate of $E[\|\hat{\beta}-\beta\|^2] $.
Recall $m(y)$ is given by \eqref{my.1}.
Then 
\begin{equation}\label{hat.R}
\hat{R}(\|y\|^2)=
p-2\frac{ \int_{0}^{1}\kappa^{p/2+a}(1-\kappa)^{b-1}\exp(-\kappa \|y\|^2/2) \rd \kappa}
{\int_{0}^{1}\kappa^{p/2+a-1}(1-\kappa)^{b-1}\exp(-\kappa \|y\|^2/2) \rd \kappa}\Delta(\|y\|^2/2;a,b),
\end{equation}
where 
\begin{equation}\label{eq:Delta}
 \begin{split}
\Delta(w;a,b) &=p-2w
\frac{ \int_{0}^{1}\kappa^{p/2+a+1}(1-\kappa)^{b-1}\exp(-w\kappa ) \rd \kappa}
{ \int_{0}^{1}\kappa^{p/2+a}(1-\kappa)^{b-1}\exp(-w\kappa ) \rd \kappa}\\
&\quad
+w\frac{ \int_{0}^{1}\kappa^{p/2+a}(1-\kappa)^{b-1}\exp(-w\kappa ) \rd \kappa}
{\int_{0}^{1}\kappa^{p/2+a-1}(1-\kappa)^{b-1}\exp(-w\kappa )\rd \kappa}.
\end{split}
\end{equation}
In Lemma \ref{lem:Delta} below,  $ \Delta(w;a,b)$ is represented 
through the confluent hypergeometric function defined by
\begin{equation}
 M(b,c,w)=1+\sum_{i=1}^\infty \frac{b\cdots (b+i-1)}{c\cdots (c+i-1)}\frac{w^i}{i!}.
\end{equation}
\begin{lemma}\label{lem:Delta}
Let $0<b<1$ and $ p/2+a>0$. Then
\begin{equation}\label{DeltaDelta}
\begin{split}
 \Delta(w;a,b) &=\frac{p}{2}-a-2+2\Bigl(\frac{p}{2}+a+1\Bigr)\frac{M(b-1,p/2+a+b+1,w)}{M(b,p/2+a+b+1,w)} \\
&\quad - \Bigl(\frac{p}{2}+a\Bigr)\frac{M(b-1,p/2+a+b,w)}{M(b,p/2+a+b,w)}.
\end{split}
\end{equation}
\end{lemma}
Proof of Lemma \ref{lem:Delta} is given in Appendix \ref{subsec:lem.Delta}.

\section{Minimaxity under the general priors}
\label{sec:minimax.general}
As in Section \ref{sec:Bayes_SURE}, the risk $ E[\|\hat{\beta}-\beta\|^2]$ 
is equal to $E[\hat{R}(\|Y\|^2)]$ where
\begin{equation}\label{hat.R.proofs}
\hat{R}(\|y\|^2)=
p-2\frac{ \int_{0}^{1}\kappa^{p/2+a}(1-\kappa)^{b-1}\exp(-\kappa \|y\|^2/2) \rd \kappa}
{\int_{0}^{1}\kappa^{p/2+a-1}(1-\kappa)^{b-1}\exp(-\kappa \|y\|^2/2) \rd \kappa}\Delta(\|y\|^2/2;a,b).
\end{equation}
In the above, $\Delta(w;a,b)$ is represented by \eqref{DeltaDelta} in Lemma \ref{lem:Delta}.
Then a sufficient condition for minimaxity, or equivalently
$E[\|\hat{\beta}-\beta\|^2]\leq p$, is given by 
\begin{equation}
 \Delta(w;a,b)\geq 0 \text{ for all }w\geq 0.
\end{equation}
Let 
\begin{align}\label{Psi.b.c}
 \Psi(b,q) =\frac{4}{3}\frac{b}{1-b}\frac{b+q+1}{b+q+2}
-2\psi\bigl\{\zeta(b,q)\bigr\}-\Bigl[\psi\bigl\{\zeta(b,q)\bigr\}\Bigr]^2,
\end{align}
where
\begin{equation}\label{tautau}
 \zeta(b,q)=\frac{2}{9}\frac{b^2}{(b+1)(b+2)}\frac{(b+q+3)(b+q+4)}{(b+q+2)^2},
\end{equation}
and
\begin{equation}\label{G}
 \psi(\zeta)=\zeta^{1/3}\left\{(1+\sqrt{1-\zeta})^{1/3}+(1-\sqrt{1-\zeta})^{1/3}\right\}.
\end{equation}
Then we have a following result.
\begin{theorem}\label{thm:main_general}
Assume $p\geq 3$,
\begin{align}
&\begin{cases}
\displaystyle -\frac{p}{2}<a\leq -\frac{3}{2}, & 0<b<1, \\[9pt]
\displaystyle-\frac{3}{2}<a<\frac{p}{2}-2, & \displaystyle\frac{8a+12}{2a+3p}\leq b<1,
\end{cases} \label{range.a}\\
& \text{ and } \ \Psi(b,p/2+a)\geq 0.\label{psi.positive}
\end{align}
Then the Bayes estimator under the prior \eqref{our.prior} is minimax.
\end{theorem}
\begin{proof}
Let $q=p/2+a$ in \eqref{DeltaDelta}. Then
\begin{equation}\label{DeltaDelta.q}
\begin{split}
& \Delta(w;a,b) \\ &=\frac{p}{2}-a-2+2(q+1)\frac{M(b-1,b+q+1,w)}{M(b,b+q+1,w)} - q\frac{M(b-1,b+q,w)}{M(b,b+q,w)}.
\end{split}
\end{equation}
Recall $ 0<b<1$. Then
\begin{equation}\label{ineq:basic}
 \begin{split}
  M(b-1,b+q+1,w)&\geq M(b-1,b+q,w) \\
\text{and}\quad M(b,b+q,w)&\geq M(b,b+q+1,w)\geq 0,
 \end{split}
\end{equation}
for all $w\geq 0$.
Let
\begin{align}
 \mathcal{Q}_1&=\{w:M(b-1,b+q,w) \geq 0 \}, \\
 \mathcal{Q}_2&=\{w:M(b-1,b+q+1,w) \geq 0 > M(b-1,b+q,w) \}, \\
 \mathcal{Q}_3&=\{w:M(b-1,b+q,w)<M(b-1,b+q+1,w) < 0 \},
\end{align}
where 
\begin{equation}
 \mathcal{Q}_1\cup\mathcal{Q}_2\cup\mathcal{Q}_3=\{w:w\geq 0\} \
\text{ and } \
\mathcal{Q}_1\cap\mathcal{Q}_2=\mathcal{Q}_1\cap\mathcal{Q}_3=\mathcal{Q}_2\cap\mathcal{Q}_3=\emptyset.
\end{equation}
By \eqref{ineq:basic}, for $ w \in \mathcal{Q}_1$, we have
\begin{gather}
M(b-1,b+q+1,w) \geq 0,\\
\frac{M(b-1,b+q+1,w)}{M(b,b+q+1,w)}\geq \frac{M(b-1,b+q,w)}{M(b,b+q,w)}\geq 0, 
\end{gather}
and hence
\begin{equation}\label{eq:case_i}
 \begin{split}
 \Delta(w;a,b) &\geq \frac{p}{2}-a-2 +\left\{2(q+1)-q\right\}\frac{M(b-1,b+q,w)}{M(b,b+q,w)} \\
&\geq p/2-a-2> 0,
\end{split}
\end{equation}
where the last inequality follows from \eqref{range.a}.
For $ w \in \mathcal{Q}_2$, 
it is clear that
\begin{equation}\label{eq:case_ii}
 \Delta(w;a,b) \geq p/2-a-2> 0,
\end{equation}
where the inequality follows from \eqref{range.a}.
By \eqref{eq:case_i} and \eqref{eq:case_ii}, the theorem follows
provided
\begin{equation}
 \min_{w\in\mathcal{Q}_3}\Delta(w;a,b) \geq 0.
\end{equation}
By the assumption \eqref{psi.positive}, it follows from Lemma \ref{lem:ratio.0} below that
\begin{equation}\label{Delta.case.iii_10}
\min_{w\geq 0} \frac{M(b-1,b+q+1,w)}{M(b,b+q+1,w)} \geq -\frac{1-b}{b+2}.
\end{equation}
Also note 
\begin{equation}\label{Delta.case.iii_0}
 -\frac{M(b-1,b+q,w)}{M(b,b+q,w)}\geq 0
\end{equation}
for $w\in \mathcal{Q}_3$. By \eqref{DeltaDelta}, \eqref{Delta.case.iii_10} 
and \eqref{Delta.case.iii_0}, we have
\begin{equation}
\begin{split}
 \min_{w\in\mathcal{Q}_3}\Delta(w;a,b)
&\geq \frac{p}{2}-a-2 + 2\bigl(\frac{p}{2}+a+1\bigr) \min_{w\in\mathcal{Q}_3}
\frac{M(b-1,b+q+1,w)}{M(b,b+q+1,w)}  \\
&\geq \frac{p}{2}-a-2- (p+2a+2)\frac{1-b}{b+2} \\
&=\frac{b(3p+2a)-(8a+12)}{2(b+2)},\label{Delta.case.iii_00}
\end{split} 
\end{equation}
which is nonnegative under \eqref{range.a}. This completes the proof.
\end{proof}
The following lemma is used in the proof above.
\begin{lemma}\label{lem:ratio.0}
Suppose $0<b<1$, $q>0$ and $\Psi(b,q)\geq 0$. Then
 \begin{equation}
\min_{w\geq 0}  \frac{M(b-1,b+q+1,w)}{M(b,b+q+1,w)}\geq -\frac{1-b}{b+2}.
 \end{equation}
\end{lemma}
Proof of Lemma \ref{lem:ratio.0} is given in Appendix \ref{sec:lem:ratio}.
\section{Minimaxity under the half-Cauchy prior (Proof of Theorem \ref{thm:main})}
\label{subsec:proof.2}
[{\bfseries Case $\bm{p\geq 11}$}]{}
The result for this case is a corollary of Theorem \ref{thm:main_general}.
For $p\geq 11$, $a=b=1/2$ satisfy \eqref{range.a}, namely,
\begin{equation}
\frac{1}{2}\in\left(-p/2,p/2-2\right) \ \text{ and } \ 
\frac{1}{2}\in\left[\frac{8(1/2)+12}{2(1/2)+3p},1\right),
\end{equation}
since
\begin{equation}
\frac{1}{2}- \frac{8(1/2)+12}{2(1/2)+3p}-=\frac{1}{2}-\frac{16}{1+3p}=
  \frac{3p-31}{2(1+3p)}>0 \text{ for }p\geq 11.
\end{equation}
Further, by Parts \ref{lem:Psi.2} and \ref{lem:Psi.1} of Lemma \ref{lem:Psi} at the end of this section, we have
\begin{equation}
 \Psi(1/2,p/2+1/2)\geq \Psi(1/2,6)>0.2>0,
\end{equation}
for $p\geq 11$, which implies \eqref{psi.positive}.
Thus the minimaxity under $p\geq 11$ with $a=b=1/2$ follows.

\smallskip

[{\bfseries Case} $\bm{p=8,9,10}$]{}
As in the proof of Theorem \ref{thm:main_general}, it is clear that
\begin{equation}
\min_{w\in\mathcal{Q}_1\cup \mathcal{Q}_2} \Delta(w;1/2,1/2) >\frac{p-1-4}{2}>0,
\end{equation}
for $p=8,9,10$.
Further, by Part \ref{lem:ratio.2} of Lemma \ref{lem:ratio} at the end of this section,
\begin{equation}
\min_{w\geq 0} \frac{M(1/2-1,p/2+2,w)}{M(1/2,p/2+2,w)}\geq -\frac{1}{8}
\end{equation}
for $p=8,9,10$. 
Then, as in \eqref{Delta.case.iii_00},
\begin{equation}
\begin{split}
& \min_{w\in\mathcal{Q}_3}\Delta(w;1/2,1/2)\\
&\geq \frac{p}{2}-\frac{1}{2}-2+ \bigl(p+2\frac{1}{2}+2\bigr)
 \min_{w\geq 0}
\frac{M(1/2-1,p/2+2,w)}{M(1/2,p/2+2,w)}  \\
&\geq \frac{p-5}{2}-\frac{p+3}{8} =\frac{3p-23}{8},
\end{split} 
\end{equation}
which is nonnegative $p=8,9,10$. This completes the proof.

\medskip

[{\bfseries Case} $\bm{p=7}$]{}
As in the proof of Theorem \ref{thm:main_general}, it is clear that
\begin{equation}
 \begin{split}
\min_{w\in\mathcal{Q}_1\cup\mathcal{Q}_2} \Delta(w;1/2,1/2) &>\frac{p-1-4}{2}>0,
 \end{split}
\end{equation}
for $p=7$. By the verified computation
\begin{equation}
\frac{M(1/2-1,p/2+2,w)}{M(1/2,p/2+2,w)}\geq -\frac{1}{10}
\end{equation}
for $w\in (0,9.6)\cup(12.1,\infty)$ or equivalently $w\in[9.6,12.1]^\complement$. 
Then we have
\begin{equation}
\begin{split}
& \min_{w\in\mathcal{Q}_3\cap [9.6,12.1]^\complement}\Delta(w;1/2,1/2)\\
&\geq \frac{p}{2}-\frac{1}{2}-2+ \bigl(p+2\frac{1}{2}+2\bigr)
 \min_{w\in[9.6,12.1]^\complement}
\frac{M(1/2-1,p/2+2,w)}{M(1/2,p/2+2,w)}  \\
&\geq \frac{p-5}{2}-\frac{p+3}{10} =\frac{4p-28}{10}=\frac{2(p-7)}{5}=0,\label{case.p.7.1}
\end{split} 
\end{equation}
for $p=7$. By Part \ref{lem:ratio.2} of Lemma \ref{lem:ratio}, we have
\begin{equation}\label{eq:lem:ratio.2}
\min_{w\geq 0}\frac{M(1/2-1,p/2+2,w)}{M(1/2,p/2+2,w)}\geq -\frac{1}{8},
\end{equation}
for $p=7$.
Further, by Part \ref{lem:ratio.4} of Lemma \ref{lem:ratio}, we have
\begin{equation}\label{eq:lem:ratio.4}
-\frac{M(1/2-1,p/2+1,w)}{M(1/2,p/2+1,w)}\geq 0.08,
\end{equation}
for $w\in[9.6,12.1]$. 
Then, by \eqref{eq:lem:ratio.2} and \eqref{eq:lem:ratio.4}, we have
\begin{equation}
\begin{split}
& \min_{w\in\mathcal{Q}_3\cap[9.6,12.1]}\Delta(w;1/2,1/2)\\
&\geq \frac{p}{2}-\frac{1}{2}-2+ \bigl(p+2\frac{1}{2}+2\bigr)
 \min_{w\geq 0}
\frac{M(1/2-1,p/2+2,w)}{M(1/2,p/2+2,w)}\\ 
&\quad +\bigl(\frac{p}{2}+\frac{1}{2}\bigr)\min_{w\in[9.6,12.1]}
\left\{-\frac{M(1/2-1,p/2+1,w)}{M(1/2,p/2+1,w)}\right\}  \\
&\geq \frac{7-5}{2}-\frac{7+3}{8}+\frac{7+1}{2}\times 0.08 =0.07>0.\label{case.p.7.2}
\end{split} 
\end{equation}
Hence it follows from \eqref{case.p.7.1} and \eqref{case.p.7.2} that
\begin{equation}
\min_{w\in\mathcal{Q}_3}\Delta(w;1/2,1/2)\geq 0.
\end{equation}
This completes the proof.

\medskip

The following lemmas are used in the proof above.
\begin{lemma}\label{lem:Psi}
\begin{enumerate}
\item \label{lem:Psi.2} 
Suppose that $\Psi(b_*,q_*)\geq 0$ for fixed $b_*\in(0,1)$ and $q_*>0$.
Then $ \Psi(b_*,q)\geq 0$ follows for $q\geq q_*$.
\item \label{lem:Psi.1} $\Psi(1/2,6)>0.2$.
\end{enumerate} 
\end{lemma}
\begin{proof}
 See Appendix \ref{sec:lem:Psi}.
\end{proof}
\begin{lemma}\label{lem:ratio}
\begin{enumerate}
\item\label{lem:ratio.2}For $p=7,8,9,10$,
\begin{equation}
 \min_{w\geq 0} \frac{M(1/2-1,p/2+2,w)}{M(1/2,p/2+2,w)}\geq -\frac{1}{8}.
\end{equation}
\item\label{lem:ratio.3}For $p=7$,
\begin{equation}
 \frac{M(1/2-1,p/2+2,w)}{M(1/2,p/2+2,w)}\geq -\frac{1}{10},\text{ for }w\in(0,9.6)\cup(12.1,\infty).
\end{equation}
 \item\label{lem:ratio.4}For $p=7$,
\begin{equation}
-\frac{M(1/2-1,p/2+1,w)}{M(1/2,p/2+1,w)}\geq 0.08,\text{ for }w\in[9.6,12.1].
\end{equation}
\end{enumerate}
\end{lemma}
\begin{proof}
 See Appendix \ref{sec:lem:ratio}.
\end{proof}
The proofs of Part \ref{lem:Psi.1} of Lemma \ref{lem:Psi}
and Parts \ref{lem:ratio.2}--\ref{lem:ratio.4} of Lemma \ref{lem:ratio} 
utilize the interval arithmetic \citep{Moore-2009}, which has been the main tool of verified computation in the
area of computer science. See Appendix~\ref{sec:interval}.

\appendix
\label{sec:appendx}
\section{Proof of Lemma \ref{lem:Delta}}
\label{subsec:lem.Delta}
Let $q=p/2+a$. Then the function $\Delta(w;a,b)$ given by \eqref{eq:Delta}
is re-expressed as
\begin{equation}\label{eq:Delta.1}
 \begin{split}
\Delta(w;a,b) &=p-2w
\frac{ \int_{0}^{1}\kappa^{q+1}(1-\kappa)^{b-1}\exp(w\{1-\kappa\}) \rd \kappa}
{ \int_{0}^{1}\kappa^{q}(1-\kappa)^{b-1}\exp(w\{1-\kappa\} ) \rd \kappa}\\
&\quad
+w\frac{ \int_{0}^{1}\kappa^{q}(1-\kappa)^{b-1}\exp(w\{1-\kappa\}) \rd \kappa}
{\int_{0}^{1}\kappa^{q-1}(1-\kappa)^{b-1}\exp(w\{1-\kappa\}) \rd \kappa}.
\end{split}
\end{equation}
In the denominator of the second and third terms of the right-hand side of \eqref{eq:Delta.1}, 
 we have
\begin{equation}
 \begin{split}
&\int_{0}^{1}\kappa^{q+j-1}(1-\kappa)^{b-1}\exp(w\{1-\kappa\})\rd \kappa 
=\sum_{i=0}^\infty \frac{w^i}{i!}\int_{0}^{1}\kappa^{q+j-1}(1-\kappa)^{b+i-1}\rd \kappa \\
&=\sum_{i=0}^\infty \frac{w^i}{i!}B(q+j,b+i) 
=B(q+j,b) \sum_{i=0}^\infty \frac{w^i}{i!}\frac{B(q+j,b+i)}{B(q+j,b)} \\
&=B(q+j,b) M(b,b+q+j,w),\label{eq:confluent.hypergeo.1}
\end{split}
\end{equation}
for $j=0,1$. 
For the numerator,  
the second and third terms of the right-hand side of \eqref{eq:Delta.1}, 
note
\begin{equation}\label{key.derivative}
\frac{\rd }{\rd \kappa}\left\{-\exp(w\{1-\kappa\})+1 \right\}=w \exp(w\{1-\kappa\})
\end{equation}
and
\begin{equation}
 \left[\kappa^{q+j}(1-\kappa)^{b-1}\left\{-\exp(w\{1-\kappa\})+1 \right\}\right]_0^1=0,
\end{equation}
for $0<b<1$, $q>0$ and $j=0,1$. Then
an integration by parts gives
\begin{equation}\label{int.part.1}
 \begin{split}
& w \int_{0}^{1}\kappa^{q+j}(1-\kappa)^{b-1}\exp(w\{1-\kappa\}) \rd \kappa\\
&=(q+j)\int_{0}^{1}\kappa^{q+j-1}(1-\kappa)^{b-1}\left\{\exp(w\{1-\kappa\})-1\right\} \rd \kappa \\
&\quad -(b-1)\int_{0}^{1}\kappa^{q+j}(1-\kappa)^{b-2}\left\{\exp(w\{1-\kappa\})-1\right\} \rd \kappa\\
&=(q+j)B(q+j,b) \{M(b,b+q+j,w)-1\} \\
&\quad -(b-1)\int_{0}^{1}\kappa^{q+j}(1-\kappa)^{b-2}\left\{\exp(w\{1-\kappa\})-1\right\} \rd \kappa,
\end{split}
\end{equation}
where the last equality follows from \eqref{eq:confluent.hypergeo.1}.
Further we have
 \begin{align}
&\int_{0}^{1}\kappa^{q+j}(1-\kappa)^{b-2}\{\exp(w\{1-\kappa\})-1\}\rd \kappa \label{int.part.2}\\
&=\sum_{i=1}^\infty \frac{w^i}{i!}\int_{0}^{1}\kappa^{q+j}(1-\kappa)^{b+i-2}\rd \kappa \\
&=\sum_{i=1}^\infty \frac{w^i}{i!}B(q+1+j, b+i-1) \\
&=\frac{(q+j)B(q+j,b)}{b+q+j}
\left(w+\sum_{i=2}^\infty \frac{b\cdots (b+i-2)}{(b+q+j+1)\cdots (b+q+i+j-1)}\frac{w^i}{i!}\right) \\
&=\frac{(q+j)B(q+j,b)}{b-1}
\sum_{i=1}^\infty \frac{(b-1)\cdots (b+i-2)}{(b+q+j)\cdots (b+q+i+j-1)}\frac{w^i}{i!}\\
  &=\frac{(q+j)B(q+j,b)}{b-1}\left\{M(b-1,b+q+j,w)-1\right\}.
\end{align}
By \eqref{int.part.1} and \eqref{int.part.2}, we have
\begin{equation}\label{int.part.3}
 \begin{split}
  & w \int_{0}^{1}\kappa^{q+j}(1-\kappa)^{b-1}\exp(w\{1-\kappa\}) \rd \kappa\\
&=(q+j)B(q+j,b)\left\{M(b,b+q+j,w)-M(b-1,b+q+j,w)\right\}.
 \end{split}
\end{equation}
By \eqref{eq:Delta.1}, \eqref{eq:confluent.hypergeo.1} and \eqref{int.part.3},
we have
\begin{equation}
 \begin{split}
&\Delta(w;a,b)\label{Delta.general}\\
&=p-q-2+2(q+1)\frac{M(b-1,b+q+1,w)}{M(b,b+q+1,w)} 
- q\frac{M(b-1,b+q,w)}{M(b,b+q,w)} \\
&=\frac{p}{2}-a-2+2(q+1)\frac{M(b-1,b+q+1,w)}{M(b,b+q+1,w)} 
- q\frac{M(b-1,b+q,w)}{M(b,b+q,w)}.
\end{split}
\end{equation}

\section{Proof of Lemmas \ref{lem:ratio.0} and \ref{lem:ratio}}
\label{sec:lem:ratio}
Lemma \ref{lem:ratio.0} and 
Parts \ref{lem:ratio.2} and \ref{lem:ratio.3} of Lemma \ref{lem:ratio} 
are proved through the following expression,
\begin{equation}
 M(b-1,b+q+1,w)+\delta M(b,b+q+1,w)=f_N(w;\delta)+g_N(w;\delta)
\end{equation}
where
\begin{equation}
\begin{split}
f_N(w;\delta)&=(1+\delta) 
+\frac{b-1+\delta b}{b+q+1}w \\ &\quad +
\sum_{i=2}^N \{(b-1+\delta(b+i-1))\}\frac{b\cdots(b+i-2)}{(b+q+1)\cdots (b+q+i)}\frac{w^i}{i!} \\
g_N(w;\delta)&=\sum_{i=N+1}^\infty \{(b-1+\delta(b+i-1))\}\frac{b\cdots(b+i-2)}{(b+q+1)\cdots (b+q+i)}\frac{w^i}{i!} .
\end{split} 
\end{equation}
In the following proofs, we choose $\delta$ such that
\begin{equation}
\frac{1-b}{b+N-1}< \delta <\frac{1-b}{b}.
\end{equation}
Since we have
\begin{equation}
b-1+\delta(b+i-1)>0 \text{ for }i\geq N+1\text{ and }\delta\in\left(\frac{1-b}{b+N-1},\frac{1-b}{b}\right),
\end{equation}
$ g_N(w;\delta)\geq 0$ for all $w\geq 0$ follows.
If $f_N(w;\delta)\geq 0$ is also satisfied, we can conclude that
\begin{equation}
 \frac{M(b-1,b+q+1,w)}{M(b,b+q+1,w)}\geq -\delta .
\end{equation}
In the proofs, we will focus on the sufficient condition for $f_N(w;\delta)\geq 0$.
Lemma \ref{lem:unique.minimum} below guarantees that $f_N(w;\delta)$ takes
a unique minimum value on $(0,\infty)$.

\medskip
[Lemma \ref{lem:ratio.0}] 
Let
\begin{equation}
N=4\ \text{ and } \ \delta=\frac{1-b}{b+2}.
\end{equation}
For $f_4(w;\delta)$, we have
\begin{equation}
\begin{split}
& 1+\delta=\frac{3}{1-b}\delta,\quad b-1+\delta b=-2\delta,\quad 
b\{b-1+\delta (b+1)\}=-b\delta, \\
&  b(b+1)(b+2)\{(b-1)+\delta (b+3)\}=b(b+1)(b+2) \delta,
\end{split} 
\end{equation}
and hence
\begin{equation}\label{fgamma3}
\begin{split}
 \frac{f_4(w;\delta)}{\delta}-\frac{3}{1-b}
&=-\frac{2}{b+q+1}w-\frac{b}{(b+q+1)(b+q+2)}\frac{w^2}{2}\\
&\quad + \frac{b(b+1)(b+2)}{(b+q+1)(b+q+2)(b+q+3)(b+q+4)}\frac{w^4}{4!} \\
&=\frac{2}{b+q+1}\left(-w- \frac{b}{2(b+q+2)}
\frac{w^2}{2}\right. \\
&\left. \quad +
 \frac{b(b+1)(b+2)}{2(b+q+2)(b+q+3)(b+q+4)}
\frac{w^4}{4!}
\right).
\end{split} 
\end{equation}
For $b\in(0,1)$ and $q>0$, $\zeta(b,q)$ defined in \eqref{tautau} is bounded as
\begin{equation}
\begin{split}
 \zeta(b,q)&=\frac{2}{9}\frac{b^2}{(b+1)(b+2)}\frac{(b+q+3)(b+q+4)}{(b+q+2)^2}\\
&<\frac{2}{9}\frac{1^2}{(1+1)(1+2)}\frac{(0+0+3)(0+0+4)}{(0+0+2)^2}=\frac{1}{9}.
\end{split} 
\end{equation}
Then Lemma \ref{thm:Cardano} below gives 
\begin{equation}
 \min_{w\geq 0}\frac{f_4(w;\delta)}{\delta}-\frac{3}{1-b}
=  -
\frac{9(b+q+2)}{4b(b+q+1)}
\Bigl(2\psi(\zeta(b,q))+\{\psi(\zeta(b,q))\}^2\Bigr),
\end{equation}
which completes the proof.

\medskip
[Part \ref{lem:ratio.2} of Lemma \ref{lem:ratio}]  
Let $N=8$ and $\delta=1/8$. Recall $r=p/2+a+b=p/2+1$ in this case.
Then 
\begin{equation}
\begin{split}
& f_8(w;\delta) \\ &=(1+\delta)+\frac{b-1+\delta b}{r+1}w
+\sum_{i=2}^8
\{(b-1)+\delta (b+i-1)\}\frac{b\dots(b+i-2)}{(r+1)\dots(r+i)}\frac{w^i}{i!}\\
&=\frac{9}{8}-\frac{7}{8(p+4)}w
+\sum_{i=2}^8\frac{-7+2(i-1)}{16}\frac{(1/2)\dots(1/2+i-2)}{(p/2+2)\dots(p/2+1+i)}\frac{w^i}{i!}.
\end{split}
\end{equation}
The positivity of $f_8(w;\delta)$ for $w\geq 0$ with $p=7,8,9,10$ 
follows from the verified computation.

\medskip
[Part \ref{lem:ratio.3} of Lemma \ref{lem:ratio}]  
Let $p=7$, $N=20$ and $\delta=1/10$. Recall $r=p/2+a+b=p/2+1=9/2$ in this case.
Then 
\begin{equation}
\begin{split}
& f_{20}(w;\delta) \\ &=(1+\delta)+\frac{b-1+\delta b}{r+1}w
+\sum_{i=2}^{20}
\{(b-1)+\delta (b+i-1)\}\frac{b\dots(b+i-2)}{(r+1)\dots(r+i)}\frac{w^i}{i!}\\
&=\frac{11}{10}-\frac{9}{110}w
+\sum_{i=2}^{20}\frac{-9+2(i-1)}{20}
\frac{(1/2)\dots(1/2+i-2)}{(11/2)\dots(9/2+i)}\frac{w^i}{i!}.
\end{split}
\end{equation}
The positivity of $f_{20}(w;\delta)$ for 
\begin{equation}
 w\in(0,9.6)\cup(12.1,\infty)
\end{equation}
with $p=7$ 
follows from the verified computation.

\medskip
[Part \ref{lem:ratio.4} of Lemma \ref{lem:ratio}]  
Let $L$ be an integer strictly greater than $1$.
For the numerator, we have
\begin{equation}\label{numnum}
 \begin{split}
-M(-1/2,p/2+1,w)&=-1-\sum_{i=1}^\infty
\frac{(-1/2)\dots(-3/2+i)}{(p/2+1)\dots(p/2+i)}\frac{w^i}{i!} \\
&\geq -1-\sum_{i=1}^L
\frac{(-1/2)\dots(-3/2+i)}{(p/2+1)\dots(p/2+i)}\frac{w^i}{i!}, 
\end{split}
\end{equation}
which is an $L$-th polynomial with respect to $w$.
Further, for the denominator, we have 
\begin{equation}\label{denden}
 \begin{split}
&  \sum_{i=L +1}^\infty
\frac{(1/2)\dots(1/2+i-1)}{(p/2+1)\dots(p/2+i)}\frac{w^i}{i!}
\leq \frac{(1/2)\dots(1/2+L)}{(p/2+1)\dots(p/2+L+1)}\sum_{i=L+1}^\infty
\frac{w^i}{i!}\\ 
&=\frac{(1/2)\dots(1/2+L)}{(p/2+1)\dots(p/2+L+1)}
\left(\exp(w)-1- \sum_{i=1}^L
\frac{w^i}{i!}\right).
 \end{split}
\end{equation}
Hence for $w\leq w_u$, we have
\begin{equation}\label{denden.1}
 \begin{split}
&M(1/2,p/2+1,w)\\
&=1+\left\{\sum_{i=1}^L+\sum_{i=L+1}^\infty\right\}
\frac{(1/2)\dots(1/2+i-1)}{(p/2+1)\dots(p/2+i)}\frac{w^i}{i!}
\\
&\leq 1+
\frac{(1/2)\dots(1/2+L)}{(p/2+1)\dots(p/2+L+1)}
\left(\exp(w_u)-1\right) \\ &\quad +\sum_{i=1}^L
\left(
\frac{(1/2)\dots(1/2+i-1)}{(p/2+1)\dots(p/2+i)}
-\frac{(1/2)\dots(1/2+L)}{(p/2+1)\dots(p/2+L+1)}
\right)
\frac{w^i}{i!},
\end{split}
\end{equation}
which is an $L$-th polynomial with respect to $w$.
From the verified computation with \eqref{numnum}, \eqref{denden}, \eqref{denden.1}, $w_u=12.1$, $p=7$ and $L=20$,
we have
\begin{equation}
-\frac{M(1/2-1,p/2+1,w)}{M(1/2,p/2+1,w)}\geq 0.08\text{ for }w\in[9.6,12.1].
\end{equation}

\section{Preliminary results for the proof of Lemmas \ref{lem:ratio.0} and \ref{lem:ratio}}
\begin{lemma}\label{lem:unique.minimum}
 Suppose $\ell,m\in\mathbb{N}$, $m\geq 2$, $1\leq \ell\leq m-1$.
Let 
\begin{equation}
 f(x)=-\sum_{i=1}^\ell \alpha_i x^i+\sum_{i=\ell+1}^m \alpha_i x^i,
\end{equation}
where
\begin{equation}
\begin{cases}
\alpha_i\geq 0 \text{ for } 1\leq i\leq \ell\text{ with }\alpha_1>0 \text{ and } \alpha_\ell>0, \ \\
\alpha_i\geq 0 \text{ for } \ell+1\leq i\leq m\text{ with }\alpha_m>0.
\end{cases} 
\end{equation}
Then $f(x)$ has a unique extreme minimum value on $(0,\infty)$.
\end{lemma}
\begin{proof}
We have
\begin{equation}
 f'(x)=-\sum_{i=1}^\ell i \alpha_i x^{i-1}+\sum_{i=\ell+1}^m i\alpha_i x^{i-1}
\end{equation}
and
\begin{equation}\label{eq:g''_0}
 f''(x)=
\begin{cases}
\displaystyle -\sum_{i=2}^\ell i(i-1) \alpha_i x^{i-1}+\sum_{i=\ell+1}^m i(i-1)\alpha_i x^{i-2} & \text{ for }\ell\geq 2, \\
\displaystyle \hspace{9em} \sum_{i=\ell+1}^m i(i-1)\alpha_i x^{i-2} & \text{ for }\ell =1.
\end{cases}
\end{equation}
For $f''(x)$ with $\ell\geq 2$, we have
\begin{equation}
 \sum_{i=2}^\ell i(i-1) \alpha_i x^{i-2}
=\frac{\ell-1}{x}\sum_{i=2}^\ell i \alpha_i x^{i-1}-\sum_{i=2}^\ell i(\ell-i) \alpha_i x^{i-2}
\end{equation} 
and
\begin{equation}
 \sum_{i=\ell+1}^m i(i-1)\alpha_i x^{i-2}
=\frac{\ell-1}{x}\sum_{i=\ell+1}^m i \alpha_i x^{i-1}+\sum_{i=\ell+1}^m i(i-\ell)\alpha_i x^{i-2}.
\end{equation}
Then, $f''(x)$ for $\ell\geq 2$ is re-expressed as
\begin{equation}\label{eq:g''}
 f''(x)=\frac{\ell-1}{x}\left(f'(x)+\alpha_1\right)+\sum_{i=2}^m i|\ell-i| \alpha_i x^{i-2}.
\end{equation}
Note $f'(0)=-\alpha_1<0$ and $\lim_{x\to\infty}f'(x)=+\infty$.
By the intermediate value theorem, there exists $x_1$ such that
\begin{equation}
 f'(x)<0\text{ for }[0,x_1)\text{ and }f'(x_1)=0.
\end{equation}

By \eqref{eq:g''_0} with $\ell=1$, we have $ f''(x)>0$ for $x\in [x_1,\infty)$ and hence
$ f'(x)>0$ for $x\in(x_1,\infty)$.

By \eqref{eq:g''}, we have $f''(x_1)>0$ for $\ell\geq 2$ and by continuity of $f'(x)$,
there exists $x_2$ such that $f'(x)>0$ for all $x\in(x_1, x_2]$.
As the assumption for proof by contradiction,
let us assume there exists $x_3(>x_2)$ such that
\begin{equation}
 f'(x)\geq 0 \text{ for }x\in[x_1,x_3)\text{ and }f'(x_3)=0.
\end{equation}
Then we have
\begin{equation}\label{integral.0}
 \int_{x_1}^{x_3}\frac{f''(x)}{f'(x)+\alpha_1}\rd x=\left[\log\{f'(x)+\alpha_1\}\right]_{x_1}^{x_3}=0.
\end{equation}
By \eqref{eq:g''}, we have
\begin{equation}
 \frac{f''(x)}{f'(x)+\alpha_1}=\frac{\ell-1}{x}+\frac{\sum_{i=2}^m i|\ell-i| \alpha_i x^{i-2}}{f'(x)+\alpha_1}
\end{equation}
for all $x\in(0,\infty)$ and
\begin{equation}\label{integral.1}
 \int_{x_1}^{x_3}\left(\frac{\ell-1}{x}+\frac{\sum_{i=2}^m i|\ell-i| \alpha_i x^{i-2}}{f'(x)+\alpha_1}\right)
\rd x
>(\ell-1)\log\frac{x_3}{x_1}>0,
\end{equation}
which contradicts \eqref{integral.0}. 
Hence, for $\ell\geq 2$, we have $f'(x)>0$ for $(x_1,\infty)$
and the result follows.
\end{proof}

\begin{lemma}\label{thm:Cardano}
Let 
\begin{equation*}
 F(x)=-x-\gamma_2\frac{x^2}{2}+\gamma_4\frac{x^4}{4!},
\end{equation*}
where $\gamma_2>0$ and 
\begin{equation}\label{hanbetsu}
9\gamma_4> 8\gamma_2^3.
\end{equation}
Let $\zeta=8\gamma_2^3/(9\gamma_4)$.
Then, 
\begin{equation*}
 \min_{x\geq 0} F(x)
 =-\frac{9}{16\gamma_2}
\Bigl(2\psi(\zeta)+\{\psi(\zeta)\}^2\Bigr),
\end{equation*}
where $\psi(\zeta)=\zeta^{1/3}\left\{(1+\sqrt{1-\zeta})^{1/3}+(1-\sqrt{1-\zeta})^{1/3}\right\}$.
\end{lemma}
\begin{proof}
The derivative of $F(x)$ is given by
\begin{equation*}
 F'(x)=-1-\gamma_2x+\gamma_4\frac{x^3}{3!}
=\frac{\gamma_4}{6}\left(-6\frac{1}{\gamma_4}-6\frac{\gamma_2}{\gamma_4}x+x^3\right).
\end{equation*}
Note \eqref{hanbetsu} is regarded as the discriminant of the cubic equation $ F'(x)=0$.
Then Cardano's formula gives
the unique real solution of $F'(x)=0$, 
\begin{equation}\label{xast}
\begin{split}
 x_*&=\left\{\frac{3}{\gamma_4}+\sqrt{\displaystyle \bigl(\frac{3}{\gamma_4}\bigr)^2
-\bigl(\frac{2\gamma_2}{\gamma_4}\bigr)^3}\right\}^{1/3}+ 
\left\{\frac{3}{\gamma_4}-\sqrt{\displaystyle \bigl(\frac{3}{\gamma_4}\bigr)^2
-\bigl(\frac{2\gamma_2}{\gamma_4}\bigr)^3}\right\}^{1/3}\\
&=\frac{3\zeta^{1/3}}{2\gamma_2}\{(1+Z)^{1/3}+(1-Z)^{1/3}\},
\end{split} 
\end{equation}
where $Z=\sqrt{1-\zeta}$ and that
\begin{equation}
 F'(x)<0 \text{ for }0\leq x<x_*, \text{ and }F'(x)>0 \text{ for } x>x_*.
\end{equation}
For $x=x_*$ given by \eqref{xast}, we have
\begin{equation}\label{Fxast}
\begin{split}
 F(x_*)&=-x_*-\gamma_2\frac{x_*^2}{2}+\gamma_4\frac{x_*^4}{4!} \\
&=-
\frac{3\zeta^{1/3}}{2\gamma_2}\{(1+Z)^{1/3}+(1-Z)^{1/3}\} \\
&\quad -\frac{\gamma_2}{2}
\frac{9\zeta^{2/3}}{4\gamma_2^2}
\{(1+Z)^{1/3}+(1-Z)^{1/3}\}^2\\
&\quad +\frac{\gamma_4}{24}
\frac{81\zeta^{4/3}}{16\gamma_2^4}
\{(1+Z)^{1/3}+(1-Z)^{1/3}\}^4. 
\end{split}
\end{equation}
 Note $ 1-Z^2=\zeta$,
\begin{equation}
\begin{split}
 \{(1+Z)^{1/3}+(1-Z)^{1/3}\}^3
&=2+3(1-Z^2)^{1/3}\{(1+Z)^{1/3}+(1-Z)^{1/3}\}\\
&=2+3\zeta^{1/3}\{(1+Z)^{1/3}+(1-Z)^{1/3}\},
\end{split} 
\end{equation}
and
\begin{equation}
 \frac{\gamma_4}{24}
\frac{81\zeta^{4/3}}{16\gamma_2^4}=\frac{3}{16}\frac{9\gamma_4}{8\gamma_2^3}\frac{\zeta^{4/3}}{\gamma_2}
=\frac{3\zeta^{1/3}}{16\gamma_2}.
\end{equation}
Then, for the last term of $F(x_*)$ given by \eqref{Fxast}, we have
\begin{equation}\label{Fxast.last}
 \begin{split}
& \frac{\gamma_4}{24}
\frac{81\zeta^{4/3}}{16\gamma_2^4}
\{(1+Z)^{1/3}+(1-Z)^{1/3}\}^4 \\
&=\frac{3\zeta^{1/3}}{8\gamma_2}
\{(1+Z)^{1/3}+(1-Z)^{1/3}\}
+\frac{9\zeta^{2/3}}{16\gamma_2}
\{(1+Z)^{1/3}+(1-Z)^{1/3}\}^2.
 \end{split}
\end{equation}
Then, by \eqref{Fxast} and \eqref{Fxast.last}, we have
\begin{equation*}
\begin{split}
 F(x_*)&=-
\frac{9\zeta^{1/3}}{8\gamma_2}
\{(1+Z)^{1/3}+(1-Z)^{1/3}\}
-\frac{9\zeta^{2/3}}{16\gamma_2}
\{(1+Z)^{1/3}+(1-Z)^{1/3}\}^2 \\
&=-\frac{9}{16\gamma_2}
\Bigl(2\psi(\zeta)+\{\psi(\zeta)\}^2\Bigr),
\end{split} 
\end{equation*}
which completes the proof.
\end{proof}

\section{Proof of Lemma \ref{lem:Psi}}
\label{sec:lem:Psi}
Recall, as in \eqref{Psi.b.c}, \eqref{tautau}, and \eqref{G},
\begin{align}\label{Psi.b.c.1}
 \Psi(b,q) =\frac{4}{3}\frac{b}{1-b}\frac{b+q+1}{b+q+2}
-2\psi\bigl\{\zeta(b,q)\bigr\}-\Bigl[\psi\bigl\{\zeta(b,q)\bigr\}\Bigr]^2,
\end{align}
where
\begin{equation}\label{tautau.1}
 \zeta(b,q)=\frac{2}{9}\frac{b^2}{(b+1)(b+2)}\frac{(b+q+3)(b+q+4)}{(b+q+2)^2},
\end{equation}
and
\begin{equation}\label{G.1}
 \psi(\zeta)
=\bigl\{\zeta(1+\sqrt{1-\zeta})\bigr\}^{1/3} + \bigl\{\zeta(1-\sqrt{1-\zeta})\bigr\}^{1/3}.
\end{equation}

[Part \ref{lem:Psi.2}] 
The first term of $\Psi(b,q)$ is increasing in $q$.
In the second and third terms of $\Psi(b,q)$, $\zeta(b,q)$ is decreasing in $q$.
Hence it suffices to show that $\psi(\zeta)$ is increasing in $\zeta$.
Note the second term of $\psi(\zeta)$ is increasing in $\zeta$. 
For the first term of $\psi(\zeta)$, the derivative of $\zeta(1+\sqrt{1- \zeta})$ is
\begin{equation}
\begin{split}
&1+\sqrt{1- \zeta}-\frac{1}{2}\frac{ \zeta}{\sqrt{1- \zeta}} 
=\frac{\sqrt{1- \zeta}+1-(3/2) \zeta}{\sqrt{1- \zeta}} \\
&\geq \frac{1- \zeta+1-(3/2) \zeta}{\sqrt{1- \zeta}} 
= \frac{1}{2}\frac{4-5 \zeta}{\sqrt{1- \zeta}} \geq 0,
\end{split} 
\end{equation}
where the inequalities follow from the fact $ \zeta\in(0, 1/9)$. This completes the proof.

[Part \ref{lem:Psi.1}] It follows from the verified computation.

\section{Interval arithmetic}\label{sec:interval}
In this paper, we employed the interval arithemetic \citep{Moore-2009} to rigorously bound the value of Stein's unbiased risk estimate.
We used the python package pyinterval (\texttt{https://github.com/taschini/pyinterval}).
Here, we briefly explain the idea of the interval arithmetic.
See \cite{Moore-2009} for more details.

In the interval arithemetic, each number is represented by an interval that includes it.
For example, $\sqrt{2}$ can be represented by $[1.41, 1.42]$.
Such a representation enables to obtain a rigorous bound of numerical computation results  accounting for the rounding error.
For example, by representing $\sqrt{2}$ and $\sqrt{3}$ by $[1.41, 1.42]$ and $[1.73, 1.74]$ respectively, $\sqrt{2}+\sqrt{3}$ is guaranteed to be included in $[1.41+1.73,1.42+1.74]=[3.14,3.16]$.
Functions of intervals are defined in a similar way.
The interval Newton method (Algorithm~\ref{alg:inewton}) outputs an interval that includes the zero point of a given function.

\begin{algorithm}[H]
	\caption{Interval Newton method}
	\label{alg:inewton}
	\begin{algorithmic}[1]
			\REQUIRE{$f: \mathbb{R} \to \mathbb{R}$ such that $f(\bar{x})=0$ \& $(l_0,u_0)$ such that $\bar{x} \in [l_0, u_0]$ \& MAX\_ITER}
			\ENSURE{$(l,u)$ such that $\bar{x} \in [l,u]$}
			\STATE{$l \gets l_0$}
			\STATE{$u \gets u_0$}
			\FOR{$i=1,\dots,\mathrm{MAX\_ITER}$}
			\STATE{$m \gets (l+u)/2$}
			\STATE{$[a,b] \gets m-f(m)/f'([l,u])$ \% interval division}
			\STATE{$l \gets \max(l,a)$}
			\STATE{$u \gets \min(u,b)$}
			\ENDFOR
		\end{algorithmic}
\end{algorithm}

\end{document}